\documentclass[12pt,a4paper]{article}
\usepackage[utf8]{inputenc}
\usepackage{amsmath}
\usepackage{breqn}
\usepackage{amsthm}
\usepackage{amsfonts}
\usepackage{amssymb} 
\usepackage{graphicx}
\usepackage[left=2cm,right=2cm,top=2cm,bottom=2cm]{geometry}
\usepackage[numbers]{natbib}
\usepackage{lineno,hyperref}
\modulolinenumbers[5]

\newtheorem{theorem}{Theorem}[section]
\newtheorem{corollary}{Corollary}[section]
\newtheorem{lemma}{Lemma}[section]
\newtheorem{proposition}{Proposition}[section]

\newtheorem*{keyword}{Keywords}
\newtheorem*{remark}{Remark}

\newtheorem*{definition}{Definition}

\begin{document}
\title{Explicit Formula for the $n$-th Derivative of a Quotient}
\author{Roudy El Haddad \\
Universit\'e La Sagesse, Facult\'e de g\'enie, Polytech \\
\href{roudy1581999@live.com}{roudy1581999@live.com}}
\date{}
\maketitle

\begin{abstract}
Leibniz's rule for the $n$-th derivative of a product is a very well known and extremely useful formula. 
In this article, we introduce an analogous explicit formula for the $n$-th derivative of a quotient of two functions. Later, we use this formula to derive new partition identities and to develop expressions for some special $n$-th derivatives. 
\end{abstract}

\begin{keyword}
{\em $n$-th derivative of a quotient, generalized quotient rule, partitions. } \\
\end{keyword}

\section{Introduction} \label{intro}
If we chose two functions $u$ and $v$ and went around asking mathematicians to compute the $n$-th derivative of their product, the first idea that would come to their mind is to use Leibniz's formula. However, what if instead we asked them to compute the $n$-th derivative of the quotient, what formula would come to their mind? For many, the answer is none. A large portion of the mathematical community is unaware of the existence of such a formula.  
This is because, while Leibniz's formula is a subject that is studied in practically all calculus courses, when it comes to talk about an analogous formula for the quotient of two functions, the topic is rarely discussed. Although many wonder if such a formula exists, not much work has been done on the subject. In 1967 \cite{OneOverf(x)}, the first step was taken as a simpler question was answered, that is, a recursive formula for the $n$-th derivative of $1/f(x)$ was presented. Later, in 1980, F. Gerrish \cite{uselessFormula} noticed an interesting pattern linking the $n$-th derivative of a quotient to a special determinant. In 2008, this special connection was used to establish a recursive formula for such a derivative \cite{quotientFormula}. 
So if such a formula exists, why do most of us not know about it? There are two major reasons: The first is that the existing formulas are not very practical as they are recursive rather than explicit. The second is that such a formula was thought to be useless. 
F. Gerrish \cite{uselessFormula} even went as far as calling it the ``Useless Formula''. However, since then, this formula has found a variety of applications and has been used to deal with a multitude of topics \cite{furrer, barabesi2020, liu2014, mahmudov2017, rafeiro2010, cao2017, basunew}. Hence, in this article, we propose to revisit the subject and develop an explicit formula for the $n$-th derivative of a quotient analogous to the generalized product rule, in hopes this formula will become a standard like Leibniz's rule.
More precisely, it will be analogous to the generalized product rule for the product of several functions. 
Note that we mean by analogous that the formula will be explicit and that it will have the same form (that is, it will be in terms of a sum over partitions).  
Let us begin by noting that, in the same way Leibniz's formula is often referred to as the product rule, in this article, for simplicity, we will refer to the formula for the $n$-th derivative of a quotient as the quotient rule. 
We will begin by deriving a new formula for the $n$-th derivative of $1/f(x)$ (Section \ref{deriv One Over f(x)}). Although, such a formula already exists, the formula presented in \cite{OneOverf(x)} is rather complicated. We propose a simpler formula involving partitions. The formula we will present also has the advantage of being explicit rather than recursive. We will refer to this particular case as the reciprocal rule. Similarly, although a recursive formula already exists for the quotient rule, no explicit formula exists. Therefore, in Section \ref{derivative u over v}, by combining the reciprocal formula with Leibniz's formula, we develop an explicit formula for the $n$-th derivative of the quotient of two functions. Finally, in Section \ref{Applications}, we apply the reciprocal and quotient rules developed to derive new partition identities as well as expressions for some special $n$-th derivatives. 

\section{n-th derivative of $1/v(x)$ (Reciprocal rule)} \label{deriv One Over f(x)}
We begin by introducing the concept of partitions as partitions are an essential part of the quotient rule we will develop. As defined by the author in \cite{RecurrentSums, MultipleSums}, a partition can be defined as follows: 
\begin{definition}
A partition of a non-negative integer $m$ is a set of positive integers whose sum equals $m$.
We can represent a partition of $m$ as an ordered set $(y_{k,1},\ldots,y_{k,m})$ that verifies
\begin{equation}
y_{k,1}+2y_{k,2}+ \cdots + my_{k,m}
=\sum_{i=1}^{m}{i\,y_{k,i}}
=m.
\end{equation}
\end{definition} 
\noindent The coefficient $y_{k,i}$ is the multiplicity of the integer $i$ in the $k$-th partition of $m$. Note that $0\leq y_{k,i}\leq m$ while $1\leq i \leq m$. Also note that the number of partitions of an integer $m$ is given by the partition function denoted $p(m)$ and hence, $1 \leq k \leq p(m)$.
In the remainder of this text, the subscript $k$ will be added to indicate that a given parameter is associated with a given partition. Similarly, for simplicity, we will omit the bounds of $i$ and write $\sum{iy_{k,i}}=m$ and $\sum{y_{k,i}}$. Furthermore, we define the following partition notation: 
\begin{equation} \label{pi_k}
\pi_k=\sum{iy_{k,i}},
\end{equation}
\begin{equation} \label{r_k}
r_k=\sum{y_{k,i}}.
\end{equation}
As partitions are not the main focus of this article, we will not go into more details. For readers interested in a more in-depth explanation about partitions, see \cite{andrews1998theory}. \\

Before we begin proving the main results of this section, let us introduce the following notation: In the remainder of this article, the letters $u$ and $v$ will be used to indicate a function of $x$. In other words, $u$ represents $u(x)$ and $v$ represents $v(x)$. 
\begin{definition}
Let us define the following shorthand notation: 
\begin{equation}
\left(v\right)^{(n)}=v^{(n)}=\frac{d^n}{dx^n}\left(v(x)\right). 
\end{equation}
\end{definition}
In order to prove the reciprocal rule, we need to first prove the following lemma. 
\begin{lemma} \label{l1}
We have that 
\begin{equation*}
\begin{split}
\sum_{j=0}^{n-1}{\binom{\sum{Y_{k,i}}-1}{Y_{k,1}, \ldots, Y_{k,{n-j}}-1, \ldots, Y_{k,{n}}}}
=\sum_{j=1}^{n}{\binom{\sum{Y_{k,i}}-1}{Y_{k,1}, \ldots, Y_{k,{j}}-1, \ldots, Y_{k,{n}}}}
=\binom{\sum{Y_{k,i}}}{Y_{k,1}, \ldots, Y_{k,{n}}}
.\end{split}
\end{equation*}
\end{lemma}
\begin{proof}
\begin{equation*}
\begin{split}
\sum_{j=0}^{n-1}{\binom{\sum{Y_{k,i}}-1}{Y_{k,1}, \ldots, Y_{k,{n-j}}-1, \ldots, Y_{k,{n}}}}
&=\sum_{j=0}^{n-1}{\frac{(\sum{Y_{k,i}}-1)!}{Y_{k,1}! \cdots Y_{k,{n-j}}! \cdots Y_{k,{n}}!}(Y_{k,{n-j}})} \\
&=\frac{(\sum{Y_{k,i}}-1)!}{Y_{k,1}! \cdots Y_{k,{n}}!}\sum_{j=1}^{n}{(Y_{k,{j}})} \\ 
&=\frac{(\sum{Y_{k,i}})!}{Y_{k,1}! \cdots Y_{k,{n}}!}=\binom{\sum{Y_{k,i}}}{Y_{k,1}, \ldots, Y_{k,{n}}}
.\end{split}
\end{equation*}
\end{proof}
Using the recursive formula for the quotient rule \cite{quotientFormula}, we derive the reciprocal rule.  
\begin{theorem}[Reciprocal rule] \label{derivative 1/u}
Let $v$ be an $n$ times differentiable function of $x$, for any $n\in\mathbb{N}$ and at every point where $v\neq0$, we have that 
\begin{equation*}
\left(\frac{1}{v}\right)^{(n)}
=\frac{d^n}{dx^n}\left(\frac{1}{v}\right)
=n!\sum_{\sum{iy_{k,i}}=n}{\binom{\sum{y_{k,i}}}{y_{k,1}, \ldots, y_{k,n}}\frac{(-1)^{\sum{y_{k,i}}}}{v^{\sum{y_{k,i}}+1}}\prod_{i=1}^{n}{\left[\frac{v^{(i)}}{i!}\right]^{y_{k,i}}}}.
\end{equation*}
\end{theorem}
\begin{remark}
{\em A very interesting and compact way of rewriting this theorem is as follows: }
\begin{equation*}
\left(\frac{1}{v}\right)^{(n)}
=n!\sum_{\sum{iy_{k,i}}=n}{C_{k}\prod_{i=1}^{n}{\frac{1}{y_{k,i}!}\left[\frac{v^{(i)}}{i!}\right]^{y_{k,i}}}}
\end{equation*}
where 
\begin{equation} \label{C_k}
C_k
=\frac{d^{\sum{y_{k,i}}}}{dv^{\sum{y_{k,i}}}}\left(\frac{1}{v}\right)
=\frac{(\sum{y_{k,i}})!(-1)^{\sum{y_{k,i}}}}{v^{\sum{y_{k,i}}+1}}
=\frac{(-1)^{r_k}r_k!}{v^{r_k+1}}
.\end{equation}
{\em As we can see, the general reciprocal rule using the $C_k$ notation is simple and easy to memorize. }
\end{remark}
\begin{remark}
{\em Let us define the notation $\{a\}_b$ corresponds to writing $b$ times the value $a$. Let $I_k=(\{1\}_{y_{k,1}}, \ldots, \{n\}_{y_{k,n}})$. Similarly, let $P_k=(y_{k,1}, \ldots, y_{k,n})$. Other interesting ways of writing the theorem are: }
\begin{equation*}
\left(\frac{1}{v}\right)^{(n)}
=\sum_{\sum{iy_{k,i}}=n}{\binom{n}{I_k} C_{k}\prod_{i=1}^{n}{\frac{\left[v^{(i)}\right]^{y_{k,i}}}{y_{k,i}!}}}
=\frac{1}{v}\sum_{\sum{iy_{k,i}}=n}{\binom{n}{I_k}\binom{\sum{y_{k,i}}}{P_k} \prod_{i=1}^{n}{\left[-\frac{v^{(i)}}{v}\right]^{y_{k,i}}}}
.\end{equation*}
\end{remark}
\begin{proof}
1. Base case: verify true for $n=1$. \\
\begin{equation*}
1!\sum_{\sum{iy_{k,i}}=1}{\binom{\sum{y_{k,i}}}{y_{k,1}, \ldots, y_{k,n}}\frac{(-1)^{\sum{y_{k,i}}}}{v^{\sum{y_{k,i}}+1}}\prod_{i=1}^{n}{\left[\frac{v^{(i)}}{i!}\right]^{y_{k,i}}}}
=\binom{1}{1}\frac{(-1)^1}{v^2}\left[\frac{v^{'}}{1!}\right]^{1}
=-\frac{v'}{v^2}
=\frac{d}{dx}\left(\frac{1}{v}\right)
.\end{equation*}
\begin{remark}
{\em We can also verify true for $n=0$. It is important to note that the partition assumed to correspond to zero is $(0, 0, \ldots)$. Hence, }
\begin{equation*}
0!\sum_{\sum{iy_{k,i}}=0}{\binom{\sum{y_{k,i}}}{y_{k,1}, \ldots, y_{k,n}}\frac{(-1)^{\sum{y_{k,i}}}}{v^{\sum{y_{k,i}}+1}}\prod_{i=1}^{n}{\left[\frac{v^{(i)}}{i!}\right]^{y_{k,i}}}}
=\binom{0}{0,0, \ldots}\frac{(-1)^0}{v^1}\left(1\right)
=\frac{1}{v}
=\left(\frac{1}{v}\right)^{(0)}
.\end{equation*}
\end{remark}
\noindent 2. Induction hypothesis: assume the statement is true until $(n-1)\in\mathbb{N}$. 
\begin{equation*}
\left(\frac{1}{v}\right)^{(n-1)}
=(n-1)!\sum_{\sum{iy_{k,i}}=n-1}{\binom{\sum{y_{k,i}}}{y_{k,1}, \ldots, y_{k,{n-1}}}\frac{(-1)^{\sum{y_{k,i}}}}{v^{\sum{y_{k,i}}+1}}\prod_{i=1}^{n-1}{\left[\frac{v^{(i)}}{i!}\right]^{y_{k,i}}}}.
\end{equation*}
3. Induction step: we will show that this statement is true for $n$. \\
We have to show the following statement to be true:  
\begin{equation*}
\left(\frac{1}{v}\right)^{(n)}
=n!\sum_{\sum{iy_{k,i}}=n}{\binom{\sum{y_{k,i}}}{y_{k,1}, \ldots, y_{k,{n}}}\frac{(-1)^{\sum{y_{k,i}}}}{v^{\sum{y_{k,i}}+1}}\prod_{i=1}^{n}{\left[\frac{v^{(i)}}{i!}\right]^{y_{k,i}}}}.
\end{equation*}
$$ \\ $$ 
Using the recursive formula developped in \cite{quotientFormula} with $u=1$, we have 
\begin{equation*}
\left(\frac{1}{v}\right)^{(n)}
=\frac{(-1)n!}{v}\sum_{j=1}^{n}{\frac{v^{(n+1-j)}}{(n+1-j)!}\frac{\left(\frac{1}{v}\right)^{(j-1)}}{(j-1)!}}
=\frac{(-1)n!}{v}\sum_{j=0}^{n-1}{\frac{v^{(n-j)}}{(n-j)!}\frac{\left(\frac{1}{v}\right)^{(j)}}{j!}}.
\end{equation*}
Applying the induction hypothesis, we get 
\begin{equation*}
\begin{split}
\left(\frac{1}{v}\right)^{(n)}
&=\frac{(-1)n!}{v}\sum_{j=0}^{n-1}{\frac{v^{(n-j)}}{(n-j)!} \sum_{\sum{iy_{k,i}}=j}{\binom{\sum{y_{k,i}}}{y_{k,1}, \ldots, y_{k,{j}}}\frac{(-1)^{\sum{y_{k,i}}}}{v^{\sum{y_{k,i}}+1}}\prod_{i=1}^{j}{\left[\frac{v^{(i)}}{i!}\right]^{y_{k,i}}}}}\\
&=n!\sum_{j=0}^{n-1}{\frac{v^{(n-j)}}{(n-j)!} \sum_{\sum{iy_{k,i}}=j}{\binom{\sum{y_{k,i}}}{y_{k,1}, \ldots, y_{k,{j}}}\frac{(-1)^{\sum{y_{k,i}}+1}}{v^{\sum{y_{k,i}}+2}}\prod_{i=1}^{j}{\left[\frac{v^{(i)}}{i!}\right]^{y_{k,i}}}}}
.\end{split}
\end{equation*}
Let us define an extension $(y_{k,1},\ldots,y_{k,n})$ of $(y_{k,1},\ldots,y_{k,j})$ where $y_{k,j+1}=\cdots=y_{k,n}=0$. Hence, we can write that 
\begin{equation*}
\begin{split}
\left(\frac{1}{v}\right)^{(n)}
&=n!\sum_{j=0}^{n-1}{\frac{v^{(n-j)}}{(n-j)!} \sum_{\sum{iy_{k,i}}=j}{\binom{\sum{y_{k,i}}}{y_{k,1}, \ldots, y_{k,{n}}}\frac{(-1)^{\sum{y_{k,i}}+1}}{v^{\sum{y_{k,i}}+2}}\prod_{i=1}^{n}{\left[\frac{v^{(i)}}{i!}\right]^{y_{k,i}}}}}\\
&=n!\sum_{j=0}^{n-1}{\sum_{\sum{iy_{k,i}}+(n-j)\cdot1=n}{\binom{\sum{y_{k,i}}}{y_{k,1}, \ldots, y_{k,{n}}}\frac{(-1)^{\sum{y_{k,i}}+1}}{v^{\sum{y_{k,i}}+2}}{\left[\frac{v^{(n-j)}}{(n-j)!}\right]}\prod_{i=1}^{n}{\left[\frac{v^{(i)}}{i!}\right]^{y_{k,i}}}}}
.\end{split}
\end{equation*}
Notice that $1 \leq n-j \leq n$ as $0 \leq j \leq n-1$. Now, for all $(n-j)\in[1,n]$, let us associate with each partition $(y_{k,1}, \ldots, y_{k,n})$, the partition $(Y_{k,1}, \ldots, Y_{k,n})$ such that 
\begin{equation*}
\begin{cases}
Y_{k,i}=y_{k,i}+1, & \text{for } i=n-j, \\
Y_{k,i}=y_{k,i}, & \text{otherwise}. 
\end{cases}
\end{equation*}
Notice that $\sum{Y_{k,i}}=\sum{y_{k,i}}+1$ and that $\sum{iY_{k,i}}=n$. Hence, we can write 
\begin{equation*}
\begin{split}
\left(\frac{1}{v}\right)^{(n)}
&=n!\sum_{j=0}^{n-1}{\sum_{\sum{iY_{k,i}}=n}{\binom{\sum{Y_{k,i}}-1}{Y_{k,1}, \ldots, Y_{k,{n-j}}-1, \ldots, Y_{k,{n}}}\frac{(-1)^{\sum{Y_{k,i}}}}{v^{\sum{Y_{k,i}}+1}}\prod_{i=1}^{n}{\left[\frac{v^{(i)}}{i!}\right]^{Y_{k,i}}}}} \\
&=n!\sum_{\sum{iY_{k,i}}=n}{\frac{(-1)^{\sum{Y_{k,i}}}}{v^{\sum{Y_{k,i}}+1}}\left(\prod_{i=1}^{n}{\left[\frac{v^{(i)}}{i!}\right]^{Y_{k,i}}}\right)}\sum_{j=0}^{n-1}{\binom{\sum{Y_{k,i}}-1}{Y_{k,1}, \ldots, Y_{k,{n-j}}-1, \ldots, Y_{k,{n}}}}
.\end{split}
\end{equation*}
Applying Lemma \ref{l1} to the inner sum, we obtain 
\begin{equation*}
\left(\frac{1}{v}\right)^{(n)}
=n!\sum_{\sum{iY_{k,i}}=n}{\binom{\sum{Y_{k,i}}}{Y_{k,1}, \ldots, Y_{k,{n}}}\frac{(-1)^{\sum{Y_{k,i}}}}{v^{\sum{Y_{k,i}}+1}}\prod_{i=1}^{n}{\left[\frac{v^{(i)}}{i!}\right]^{Y_{k,i}}}}
.\end{equation*}
This concludes our proof by induction. 
\end{proof}
\begin{remark}
{\em As we can see, the reciprocal rule derived (Theorem \ref{derivative 1/u}) is very similar to the product rule for the product of several functions: }
\begin{equation}
\left(u_1 \cdots u_m \right)^{(n)}
=\sum_{\ell_1+\cdots+\ell_m=n}{\binom{n}{\ell_1,\ldots,\ell_m}\prod_{i=1}^{m}{u_i^{(\ell_i)}}}
\end{equation}
\end{remark}
There exists other alternatives for proving Theorem \ref{derivative 1/u}. In what follows, we will present a few propositions that will be useful for doing so. \\
 
First, let us prove the following useful proposition for the derivative of a product. 
\begin{proposition} \label{p1}
Let $u_1$, $\ldots$, $u_n$ be differentiable functions of $x$, we have that 
\begin{equation*}
\frac{d}{dx}\left(\prod_{i=1}^{n}{u_i}\right)
=\left(\prod_{i=1}^{n}{u_i}\right)\sum_{j=1}^{n}{\frac{u_{i}^{'}}{u_i}}
.\end{equation*}
\end{proposition}
\begin{proof}
Let $f(x)=u_1 \cdots u_n$. Taking the logarithm of both sides, we get 
\begin{equation*}
\ln f(x)
=\ln\left(\prod_{i=1}^{n}{u_i}\right)
=\sum_{i=1}^{n}{\ln u_i}. 
\end{equation*}
Differentiating both sides, we get 
\begin{equation*}
-\frac{f'(x)}{f(x)}
=-\sum_{i=1}^{n}{\frac{u_{i}^{'}}{u_i}}. 
\end{equation*}
Canceling the minus sign, we obtain the desired formula. 
\end{proof}
Now we prove the following partition identity involving a special sum of multinomial coefficients. This expression is equivalent to Lemma \ref{l1} that we used to prove Theorem \ref{derivative 1/u}. 
\begin{proposition} \label{multinomialSum}
We have that 
\begin{equation*}
\sum_{\substack{\sum{\varphi_i}=\sum{Y_i}-1 \\ \varphi_i \leq Y_i}}{\binom{\sum{Y_i}-1}{\varphi_1, \ldots, \varphi_n}}
=\binom{\sum{Y_i}}{Y_1, \ldots, Y_n}
.\end{equation*}
\end{proposition}
\begin{proof}
Let $Z_i=Y_i-\varphi_i$ for $1 \leq i \leq n$. 
\begin{equation*}
\begin{split}
\sum_{\substack{\sum{\varphi_i}=\sum{Y_i}-1 \\ \varphi_i \leq Y_i}}{\binom{\sum{Y_i}-1}{\varphi_1, \ldots, \varphi_n}}
&=\frac{(\sum{Y_i}-1)!}{Y_1! \cdots Y_n!}\sum_{\substack{\sum{\varphi_i}=\sum{Y_i}-1 \\ \varphi_i \leq Y_i}}{\frac{Y_1! \cdots Y_n!}{\varphi_1! \cdots \varphi_n!}}\\
&=\frac{(\sum{Y_i}-1)!}{Y_1! \cdots Y_n!}\sum_{\substack{\sum{Z_i}=1 \\ Z_i \geq 0}}{\frac{Y_1! \cdots Y_n!}{(Y_1-Z_1)! \cdots (Y_n-Z_n)!}}
.\end{split}
\end{equation*}
Knowing that the $Z_i$'s are non-negative integers, the only way for their sum to be equal to 1 is if one of them is equal to 1 and the others are equal to 0. Hence, we have that 
\begin{equation*}
\sum_{\substack{\sum{Z_i}=1 \\ Z_i \geq 0}}{\frac{Y_1! \cdots Y_n!}{(Y_1-Z_1)! \cdots (Y_n-Z_n)!}}
=\sum_{i=1}^{n}{\frac{Y_i!}{(Y_i-1)!}}
=\sum_{i=1}^{n}{Y_i}
.\end{equation*}
Substituting back, we obtain the proposition. 
\end{proof}
\section{n-th derivative of $u(x)/v(x)$ (Quotient rule)} \label{derivative u over v}
In this section, we combine Theorem \ref{derivative 1/u} with Leibniz's rule to obtain the general quotient rule. 
\begin{theorem}[Quotient rule] \label{derivative u/v}
Let $u$ and $v$ be $n$ times differentiable functions of $x$, for any $n\in\mathbb{N}$ and at every point where $v\neq0$, we have that 
\begin{equation*}
\begin{split}
\left(\frac{u}{v}\right)^{(n)}
=\frac{d^n}{dx^n}\left(\frac{u}{v}\right)
&=n!\sum_{\ell=0}^{n}{\frac{u^{(n-\ell)}}{(n-\ell)!}\sum_{\sum{iy_{k,i}}=\ell}{\binom{\sum{y_{k,i}}}{y_{k,1}, \ldots, y_{k,\ell}}\frac{(-1)^{\sum{y_{k,i}}}}{v^{\sum{y_{k,i}}+1}}\prod_{i=1}^{\ell}{\left[\frac{v^{(i)}}{i!}\right]^{y_{k,i}}}}}\\
&=n!\sum_{\pi_k=0}^{n}{\frac{u^{(n-\pi_k)}}{(n-\pi_k)!}{\binom{\sum{y_{k,i}}}{y_{k,1}, \ldots, y_{k,\pi_k}}\frac{(-1)^{\sum{y_{k,i}}}}{v^{\sum{y_{k,i}}+1}}\prod_{i=1}^{\pi_k}{\left[\frac{v^{(i)}}{i!}\right]^{y_{k,i}}}}}
.\end{split}
\end{equation*}
\end{theorem}
\begin{proof}
Applying Leibniz's rule to Theorem \ref{derivative 1/u}, we obtain this theorem. 
\end{proof}
\begin{remark}
{\em A much more compact way of expressing this theorem is as follows: }
\begin{equation*}
\left(\frac{u}{v}\right)^{(n)}
=n!\sum_{\ell=0}^{n}{\sum_{\sum{iy_{k,i}}=n-\ell}{C_k\left[\frac{u^{(\ell)}}{\ell!}\right]\prod_{i=1}^{n-\ell}{\frac{1}{y_{k,i}!}\left[\frac{v^{(i)}}{i!}\right]^{y_{k,i}}}}}
=n!\sum_{\pi_k=0}^{n}{{C_k\frac{u^{(n-\pi_k)}}{(n-\pi_k)!}\prod_{i=1}^{\pi_k}{\frac{1}{y_{k,i}!}\left[\frac{v^{(i)}}{i!}\right]^{y_{k,i}}}}}
\end{equation*}
{\em where $C_k$ is as defined in Eq.~(\ref{C_k}) and $\pi_k$ is as defined in Eq.~(\ref{pi_k}).} \\
{\em Using the $C_k$ and $\pi_k$ notation, we obtain a simple and easy to memorize expression for the general quotient rule that could potentially be taught to university students at the same time as the general product rule. }
\end{remark}
\begin{remark}
{\em Let $I_k=(\{1\}_{y_{k,1}}, \ldots, \{\pi_k\}_{y_{k,\pi_k}})$ and $P_k=(y_{k,1},\ldots, y_{k,\pi_k})$. We can write the following interesting but not very practical expressions: }
\begin{equation*}
\begin{split}
\left(\frac{u}{v}\right)^{(n)}
&=\sum_{\pi_k=0}^{n}{{C_{k} \binom{n}{I_k,n-\pi_k} u^{(n-\pi_k)}\prod_{i=1}^{\pi_k}{\frac{\left[v^{(i)}\right]^{y_{k,i}}}{y_{k,i}!}}}} \\
&=\sum_{\pi_k=0}^{n}{{\binom{n}{I_k,n-\pi_k}\binom{\sum{y_{k,i}}}{P_k} \frac{u^{(n-\pi_k)}}{v}\prod_{i=1}^{\pi_k}{\left[-\frac{v^{(i)}}{v}\right]^{y_{k,i}}}}}
.\end{split}
\end{equation*}
\end{remark}
\section{Applications} \label{Applications}
Let us first define some notation to simplify the expressions we will derive. For a given partition $(y_{k,1}, \ldots, y_{k,n})$ of $n$, we define the following notation: 
\begin{align}
c_k&=\prod_{i=1}^{n}{\frac{1}{i^{y_{k,i}}y_{k,i}!}},
&
\overline{c}_k&=\prod_{i=1}^{n}{\frac{(-1)^{y_{k,i}}}{i^{y_{k,i}}y_{k,i}!}}. \\
p_k&=\prod_{i=1}^{n}{\frac{1}{i!^{y_{k,i}}y_{k,i}!}},
& 
\overline{p}_k&=\prod_{i=1}^{n}{\frac{(-1)^{y_{k,i}}}{i!^{y_{k,i}}y_{k,i}!}}. \\
q_k&=\prod_{i=1}^{n}{\frac{1}{i!^{y_{k,i}}}}, 
&
\overline{q}_k&=\prod_{i=1}^{n}{\frac{(-1)^{y_{k,i}}}{i!^{y_{k,i}}}}.
\end{align}
\subsection{Partition identities} \label{Partition identities}
In this section, we will show how the quotient rule developed can be used to derive partition identities. In particular, we will derive a few special partition identities. 
\begin{proposition}
For any $n\in\mathbb{N}$, we have that 
\begin{equation*}
\sum_{\sum{iy_{k,i}}=n}{\binom{\sum{y_{k,i}}}{y_{k,1},\ldots,y_{k,n}}(-1)^{\sum{y_{k,i}}}\prod_{i=1}^{n}{\frac{1}{i!^{y_{k,i}}}}}
=\frac{(-1)^n}{n!}
.\end{equation*}
Using the notation, this proposition can be expressed as
\begin{equation*}
\sum_{\sum{iy_{k,i}}=n}{\binom{r_k}{y_{k,1},\ldots,y_{k,n}}(-1)^{r_k}q_k}
=\sum_{\sum{iy_{k,i}}=n}{\binom{r_k}{y_{k,1},\ldots,y_{k,n}}\overline{q}_k}
=\frac{(-1)^n}{n!}
.\end{equation*}
\end{proposition}
\begin{remark}
We can also rewrite it as follows:
\begin{equation*}
\sum_{\sum{iy_{k,i}}=n}{r_k!(-1)^{r_k}\prod_{i=1}^{n}{\frac{1}{i!^{y_{k,i}}y_{k,i}!}}}
=\frac{(-1)^n}{n!}
.\end{equation*}
Using the notation, this proposition can be expressed as
\begin{equation*}
\sum_{\sum{iy_{k,i}}=n}{r_k!(-1)^{r_k}p_k}
=\sum_{\sum{iy_{k,i}}=n}{r_k!\overline{p}_k}
=\frac{(-1)^n}{n!}
.\end{equation*}
\end{remark}
\begin{proof}
From Theorem \ref{derivative 1/u} with $v(x)=e^x$ and knowing that $v^{(i)}(x)=e^x$ for all $i$, we get 
\begin{equation*}
\begin{split}
\frac{d^n}{dx^n}{\left(\frac{1}{e^x}\right)}
&=n!\sum_{\sum{iy_{k,i}}=n}{\binom{\sum{y_{k,i}}}{y_{k,1}, \ldots, y_{k,n}}\frac{(-1)^{\sum{y_{k,i}}}}{(e^x)^{\sum{y_{k,i}}+1}}\prod_{i=1}^{n}{\left[\frac{e^x}{i!}\right]^{y_{k,i}}}} \\
&=n!e^{-x}\sum_{\sum{iy_{k,i}}=n}{\binom{\sum{y_{k,i}}}{y_{k,1},\ldots,y_{k,n}}(-1)^{\sum{y_{k,i}}}\prod_{i=1}^{n}{\frac{1}{i!^{y_{k,i}}}}}.
\end{split}
\end{equation*}
Noticing that 
\begin{equation*}
\frac{d^n}{dx^n}{\left(\frac{1}{e^x}\right)}
=\frac{d^n}{dx^n}{\left({e^{-x}}\right)}
=(-1)^n e^{-x},
\end{equation*}
we obtain the proposition. 
\end{proof}
\begin{proposition}
For any $n\in\mathbb{N}$ and any $m\in\mathbb{N}^*$, we have that 
\begin{equation*}
\sum_{\sum{iy_{k,i}}=n}{\binom{\sum{y_{k,i}}}{y_{k,1}, \ldots, y_{k,n}}(-1)^{\sum{y_{k,i}}}\prod_{i=1}^{n}{\left[\binom{m}{i}\right]^{y_{k,i}}}}
=(-1)^n \binom{n+m-1}{m-1}. 
\end{equation*}
\end{proposition}
\begin{proof}
Let $v(x)=x^m$, then $v^{(i)}=i!\binom{m}{i}x^{m-i}$. Hence, from Theorem \ref{derivative 1/u}, we have 
\begin{equation*}
\begin{split}
\frac{d^n}{dx^n}{\left(\frac{1}{x^m}\right)}
&=n!\sum_{\sum{iy_{k,i}}=n}{\binom{\sum{y_{k,i}}}{y_{k,1}, \ldots, y_{k,n}}\frac{(-1)^{\sum{y_{k,i}}}}{x^{m\sum{y_{k,i}}+ m}}\prod_{i=1}^{n}{\left[x^{m-i}\binom{m}{i}\right]^{y_{k,i}}}} \\
&=n!\sum_{\sum{iy_{k,i}}=n}{\binom{\sum{y_{k,i}}}{y_{k,1}, \ldots, y_{k,n}}\frac{(-1)^{\sum{y_{k,i}}}}{x^{m\sum{y_{k,i}}+m}}(x^{m\sum{y_{k,i}}-n})\prod_{i=1}^{n}{\left[\binom{m}{i}\right]^{y_{k,i}}}} \\
&=n!x^{-m-n}\sum_{\sum{iy_{k,i}}=n}{\binom{\sum{y_{k,i}}}{y_{k,1}, \ldots, y_{k,n}}(-1)^{\sum{y_{k,i}}}\prod_{i=1}^{n}{\left[\binom{m}{i}\right]^{y_{k,i}}}}.
\end{split}
\end{equation*}
Noticing that 
\begin{equation*}
\frac{d^n}{dx^n}{\left(\frac{1}{x^m}\right)}
=\frac{d^n}{dx^n}{\left({x^{-m}}\right)}
=(-1)^n n!\binom{n+m-1}{m-1}  x^{-m-n},
\end{equation*}
we obtain the proposition. 
\end{proof}
\begin{corollary}
Setting $m=n$, we get 
\begin{equation*}
\sum_{\sum{iy_{k,i}}=n}{\binom{\sum{y_{k,i}}}{y_{k,1}, \ldots, y_{k,n}}(-1)^{\sum{y_{k,i}}}\prod_{i=1}^{n}{\left[\binom{n}{i}\right]^{y_{k,i}}}}
=(-1)^n \binom{2n-1}{n-1}. 
\end{equation*}
\end{corollary}
\begin{proposition} \label{p 4.3}
For any $n\in\mathbb{N}$ and any $m\in\mathbb{N}^*$, we have that 
\begin{equation*}
\sum_{\sum{iy_{k,i}}=n}{\binom{\sum{y_{k,i}}}{y_{k,1}, \ldots, y_{k,n}}(-1)^{\sum{y_{k,i}}}\prod_{i=1}^{n}{\left[\binom{i+m-1}{m-1}\right]^{y_{k,i}}}}
=(-1)^n \binom{m}{n}. 
\end{equation*}
\end{proposition}
\begin{proof}
Let $v(x)=x^{-m}$, then $v^{(i)}=(-1)^{i}i!\binom{i+m-1}{m-1}x^{-(m+i)}$. Hence, from Theorem \ref{derivative 1/u}, we have 
\begin{equation*}
\begin{split}
\frac{d^n}{dx^n}{\left(\frac{1}{x^{-m}}\right)}
&=n!\sum_{\sum{iy_{k,i}}=n}{\binom{\sum{y_{k,i}}}{y_{k,1}, \ldots, y_{k,n}}\frac{(-1)^{\sum{y_{k,i}}}}{x^{-m\sum{y_{k,i}} - m}}\prod_{i=1}^{n}{\left[(-1)^i\binom{i+m-1}{m-1}x^{-(m+i)}\right]^{y_{k,i}}}} \\
&=n!(-1)^n\sum_{\sum{iy_{k,i}}=n}{\binom{\sum{y_{k,i}}}{y_{k,1}, \ldots, y_{k,n}}\frac{(-1)^{\sum{y_{k,i}}}(x^{-m\sum{y_{k,i}}-n})}{x^{-m\sum{y_{k,i}}-m}}\prod_{i=1}^{n}{\left[\binom{i+m-1}{m-1}\right]^{y_{k,i}}}} \\
&=(-1)^n n!x^{m-n}\sum_{\sum{iy_{k,i}}=n}{\binom{\sum{y_{k,i}}}{y_{k,1}, \ldots, y_{k,n}}(-1)^{\sum{y_{k,i}}}\prod_{i=1}^{n}{\left[\binom{i+m-1}{m-1}\right]^{y_{k,i}}}}.
\end{split}
\end{equation*}
Noticing that 
\begin{equation*}
\frac{d^n}{dx^n}{\left(\frac{1}{x^m}\right)}
=\frac{d^n}{dx^n}{\left({x^{-m}}\right)}
=n!\binom{m}{n}  x^{m-n},
\end{equation*}
we obtain the proposition. 
\end{proof}
An extremely interesting result that can be derived from Proposition \ref{p 4.3} is that for the alternating sum over partitions of multinomial coefficients. 
\begin{corollary}
Setting $m=1$, we get 
\begin{equation*}
\sum_{\sum{iy_{k,i}}=n}{\binom{\sum{y_{k,i}}}{y_{k,1}, \ldots, y_{k,n}}(-1)^{\sum{y_{k,i}}}}
=(-1)^n \binom{1}{n}
=
\begin{cases}
(-1)^n, & n=0,1,\\
0, & n \geq 2.
\end{cases}
\end{equation*}
\end{corollary}
\subsection{Special n-th derivatives} \label{Special n-th derivatives}
Because of the absence of a general quotient rule, there were many $n$-th derivatives for which we could not obtain an explicit expression. In this section, we will use the quotient rule derived to develop an expression for some of these derivatives. \\ 

The first special $n$-th derivative is that of $\log_{x}{a}$ as well as that of the reciprocal of $\ln x$. In 2014, Feng Qi \cite{qi2014} introduced the following expression for the reciprocal of $\ln x$: 
\begin{equation} \label{Feng}
\left(\frac{1}{\ln x}\right)^{(n)}
=\frac{(-1)^n}{x^n}\sum_{i=2}^{n+1}{\frac{a_{n,i}}{(\ln x)^i}},
\end{equation}
where
\begin{equation}
a_{n,2}=(n-1)!
\end{equation}
and, for $n+1\geq i \geq 3$, 
\begin{equation}
a_{n,i}=(i-1)!(n-1)!\sum_{\ell_1=1}^{n-1}{\frac{1}{\ell_1}\sum_{\ell_2=1}^{\ell_1-1}{\frac{1}{\ell_2}\cdots \sum_{\ell_{i-3}=1}^{\ell_{i-4}-1}{\frac{1}{\ell_{i-3}}\sum_{\ell_{i-2}=1}^{\ell_{i-3}-1}{\frac{1}{\ell_{i-2}}}}}}. 
\end{equation}
The expression seems simple, however, the $a_{n,i}$ terms correspond to a kind of multiple harmonic sum. Such sums are very tedious to compute, thus, making Eq.~(\ref{Feng}) a bit tedious to use. In what follows, using the general reciprocal rule, we will derive a simpler expression.  
\begin{proposition}
For any $a\in\mathbb{N^*}$, the $n$-th derivative of $\log_{x}{a}$ is given by 
\begin{equation*}
\left(\log_{x}{a}\right)^{(n)}
=\left(\frac{\ln a}{\ln x}\right)^{(n)}
=(\log_{x}{a})\frac{(-1)^n n!}{x^n}\sum_{\sum{iy_{k,i}}=n}{\frac{({\sum{y_{k,i}}})!}{(\ln x)^{\sum{y_{k,i}}}}\prod_{i=1}^{n}{\frac{1}{i^{y_{k,i}}y_{k,i}!}}}.
\end{equation*}
Using the notation, we can rewrite it as follows: 
\begin{equation*}
\left(\log_{x}{a}\right)^{(n)}
=\left(\frac{\ln a}{\ln x}\right)^{(n)}
=(\log_{x}{a})\frac{(-1)^n n!}{x^n}\sum_{\sum{iy_{k,i}}=n}{c_k\frac{r_k!}{(\ln x)^{r_k}}}
.\end{equation*}
\end{proposition}
\begin{proof}
From Theorem \ref{derivative 1/u} with $v=\ln x$, we have
\begin{equation*}
\begin{split}
\left(\log_{x}{a}\right)^{(n)}
=\left(\frac{\ln a}{\ln x}\right)^{(n)}
&=n!\frac{\ln a}{\ln x}\sum_{\sum{iy_{k,i}}=n}{\binom{\sum{y_{k,i}}}{y_{k,1}, \ldots, y_{k,n}}\frac{(-1)^{\sum{y_{k,i}}}}{(\ln x)^{\sum{y_{k,i}}}}\prod_{i=1}^{n}{\left[\frac{(\ln x)^{(i)}}{i!}\right]^{y_{k,i}}}}.
\end{split}
\end{equation*}
Knowing that, for $i \geq 1$,  
\begin{equation*}
(\ln x)^{(i)}=\frac{(-1)^{i-1} (i-1)!}{x^i},
\end{equation*}
hence, by substituting back and simplifying, we get 
\begin{equation*}
\begin{split}
\left(\log_{x}{a}\right)^{(n)}
&=n!(\log_{x}{a})\sum_{\sum{iy_{k,i}}=n}{\binom{\sum{y_{k,i}}}{y_{k,1}, \ldots, y_{k,n}}\frac{(-1)^{\sum{y_{k,i}}}}{(\ln x)^{\sum{y_{k,i}}}}\prod_{i=1}^{n}{\left[\frac{(-1)^{i-1}}{i x^i}\right]^{y_{k,i}}}} \\
&=\frac{(-1)^{n} n!}{x^n}(\log_{x}{a})\sum_{\sum{iy_{k,i}}=n}{\binom{\sum{y_{k,i}}}{y_{k,1}, \ldots, y_{k,n}}\frac{1}{(\ln x)^{\sum{y_{k,i}}}}\prod_{i=1}^{n}{\left[\frac{1}{i}\right]^{y_{k,i}}}}.
\end{split}
\end{equation*}
Replacing the multinomial coefficient by its factorial definition, we obtain this proposition. 
\end{proof}
Another special $n$-th derivative is that of $\ln{v(x)}$.
\begin{proposition}
The $n$-th derivative of $\ln v(x)$ is given by 
\begin{equation*}
\begin{split}
(\ln v)^{(n)}
&=n!\sum_{\sum{iy_{k,i}}=n}{\binom{\sum{y_{k,i}}}{y_{k,1}, \ldots, y_{k,n}}\frac{(-1)^{\sum{y_{k,i}}-1}}{(\sum{y_{k,i}})!v^{\sum{y_{k,i}}}}\prod_{i=1}^{n}{\left[\frac{v^{(i)}}{i!}\right]^{y_{k,i}}}}  \\
&=n!\sum_{\sum{iy_{k,i}}=n}{\frac{(\sum{y_{k,i}}-1)!(-1)^{\sum{y_{k,i}}-1}}{v^{\sum{y_{k,i}}}}\prod_{i=1}^{n}{\frac{1}{y_{k,i}!}\left[\frac{v^{(i)}}{i!}\right]^{y_{k,i}}}}
.\end{split}
\end{equation*}
\end{proposition}
\begin{proof}
From Theorem \ref{derivative u/v}, we have 
\begin{equation*}
\begin{split}
(\ln v)^{(n)}
&=\left(\frac{v'}{v}\right)^{(n-1)} \\
&=(n-1)!\sum_{\ell=0}^{n-1}{\frac{(v')^{(\ell)}}{\ell!}\sum_{\sum{iy_{k,i}}=n-\ell-1}{\binom{\sum{y_{k,i}}}{y_{k,1}, \ldots, y_{k,n-\ell-1}}\frac{(-1)^{\sum{y_{k,i}}}}{v^{\sum{y_{k,i}}+1}}\prod_{i=1}^{n-\ell-1}{\left[\frac{v^{(i)}}{i!}\right]^{y_{k,i}}}}} \\
&=(n-1)!\sum_{\ell=0}^{n-1}{\frac{v^{(\ell+1)}(\ell+1)}{(\ell+1)!}\sum_{\sum{iy_{k,i}}=n-\ell-1}{\binom{\sum{y_{k,i}}}{y_{k,1}, \ldots, y_{k,n-\ell-1}}\frac{(-1)^{\sum{y_{k,i}}}}{v^{\sum{y_{k,i}}+1}}\prod_{i=1}^{n-\ell-1}{\left[\frac{v^{(i)}}{i!}\right]^{y_{k,i}}}}} \\
&=(n-1)!\sum_{\ell=1}^{n}{\frac{v^{(\ell)}}{\ell!}\ell\sum_{\sum{iy_{k,i}}=n-\ell}{\binom{\sum{y_{k,i}}}{y_{k,1}, \ldots, y_{k,n-\ell}}\frac{(-1)^{\sum{y_{k,i}}}}{v^{\sum{y_{k,i}}+1}}\prod_{i=1}^{n-\ell}{\left[\frac{v^{(i)}}{i!}\right]^{y_{k,i}}}}}
.\end{split}
\end{equation*}
Similar to what was done in the proof of Theorem \ref{derivative 1/u}, we defined an extension $(y_{k,1}, \cdots, y_{k,n})$ of each partition $(y_{k,1}, \cdots, y_{k,n-\ell})$ such that $y_{k,n-\ell+1}=\cdots=y_{k,n}=0$. Now, for every $\ell\in[1,n]$, let us associate with each partition $(y_{k,1}, \ldots, y_{k,n})$, the partition $(Y_{k,1}, \ldots, Y_{k,n})$ such that 
\begin{equation*}
\begin{cases}
Y_{k,i}=y_{k,i}+1, & \text{for } i=\ell, \\
Y_{k,i}=y_{k,i}, & \text{otherwise}. 
\end{cases}
\end{equation*}
Notice that $\sum{Y_{k,i}}=\sum{y_{k,i}}+1$ and that $\sum{iY_{k,i}}=n$. Hence, we can write 
\begin{equation*}
\begin{split}
(\ln v)^{(n)}
&=(n-1)!\sum_{\ell=1}^{n}{\ell\sum_{\sum{iY_{k,i}}=n}{\binom{\sum{Y_{k,i}-1}}{Y_{k,1}, \ldots, Y_{k,\ell}-1, \ldots, Y_{k,n}}\frac{(-1)^{\sum{Y_{k,i}}-1}}{v^{\sum{Y_{k,i}}}}\prod_{i=1}^{n}{\left[\frac{v^{(i)}}{i!}\right]^{Y_{k,i}}}}} \\
&=(n-1)!\sum_{\sum{iY_{k,i}}=n}{\frac{(-1)^{\sum{Y_{k,i}}-1}}{v^{\sum{Y_{k,i}}}}\prod_{i=1}^{n}{\left[\frac{v^{(i)}}{i!}\right]^{Y_{k,i}}}}\sum_{\ell=1}^{n}{\ell\binom{\sum{Y_{k,i}-1}}{Y_{k,1}, \ldots, Y_{k,\ell}-1, \ldots, Y_{k,n}}} \\
&=(n-1)!\sum_{\sum{iY_{k,i}}=n}{\binom{\sum{Y_{k,i}}}{Y_{k,1}, \ldots, Y_{k,n}}\frac{(-1)^{\sum{Y_{k,i}}-1}}{(\sum{Y_{k,i}})!v^{\sum{Y_{k,i}}}}\prod_{i=1}^{n}{\left[\frac{v^{(i)}}{i!}\right]^{Y_{k,i}}}}\sum_{\ell=1}^{n}{\ell Y_{k,\ell}} \\
&=n!\sum_{\sum{iY_{k,i}}=n}{\binom{\sum{Y_{k,i}}}{Y_{k,1}, \ldots, Y_{k,n}}\frac{(-1)^{\sum{Y_{k,i}}-1}}{(\sum{Y_{k,i}})!v^{\sum{Y_{k,i}}}}\prod_{i=1}^{n}{\left[\frac{v^{(i)}}{i!}\right]^{Y_{k,i}}}} 
.\end{split}
\end{equation*}
\end{proof}
\bibliographystyle{apa}
\bibliography{DerivativeOfAQuotient}

\end{document}